\documentclass[reqno, 12pt]{amsart}
\makeatletter
\let\origsection=\section \def\section{\@ifstar{\origsection*}{\mysection}} 
\def\mysection{\@startsection{section}{1}\z@{.7\linespacing\@plus\linespacing}{.5\linespacing}{\normalfont\scshape\centering\S}}
\makeatother        

\usepackage{amsmath,amssymb,amsthm}
\usepackage{mathrsfs}
\usepackage{mathabx}\changenotsign
\usepackage{dsfont}

\usepackage{xcolor}
\usepackage[backref]{hyperref}
\hypersetup{
    colorlinks,
    linkcolor={red!60!black},
    citecolor={green!60!black},
    urlcolor={blue!60!black}
}

\usepackage[open,openlevel=2,atend]{bookmark}

\usepackage[abbrev,msc-links,backrefs]{amsrefs}

\usepackage{doi}

\renewcommand{\PrintDOI}[1]{\doi{#1}}

\usepackage[T1]{fontenc}
\usepackage{lmodern}
\usepackage[babel]{microtype}
\usepackage[english]{babel}

\usepackage{geometry}
\usepackage{enumitem}
\def\rmlabel{\upshape({\itshape \roman*\,})}

\def\alabel{\upshape({\itshape \alph*\,})}

\def\nlabel{\upshape({\itshape \arabic*\,})}

\def\Flabel{\upshape({\itshape F\,})}

\usepackage{tikz}
\usepackage{scalerel}

\newsavebox\vdegbox
\savebox\vdegbox{\tikz{
		\draw[black,fill=black] (90:1) circle (.35);
		\draw[black,line width=0.10cm] (210:1) circle (.30);
		\draw[black,line width=0.10cm] (330:1) circle (.30);
		\draw[opacity=0] (0:1.2) circle (0.1);
	}}

\newsavebox\vvbox
\savebox\vvbox{\tikz{
		\draw[black,line width=0.10cm] (90:1) circle (.30);
		\draw[black,fill=black] (210:1) circle (.35);
		\draw[black,fill=black] (330:1) circle (.35);
		\draw[opacity=0] (0:1.2) circle (0.1);
	}}

\newsavebox\pdegbox
\savebox\pdegbox{\tikz{
		\draw[black,line width=0.10cm] (90:1) circle (.30);
		\draw[black,fill=black] (210:1) circle (.35);
		\draw[black,fill=black] (330:1) circle (.35);
		\draw[black,line width=0.28cm ] (210:1) -- (330:1);
		\draw[opacity=0] (0:1.2) circle (0.1);
	}}

\newsavebox\vvvbox
\savebox\vvvbox{\tikz{
		\draw[black,fill=black] (90:1) circle (.35);
		\draw[black,fill=black] (210:1) circle (.35);
		\draw[black,fill=black] (330:1) circle (.35);
		\draw[opacity=0] (0:1.2) circle (0.1);
	}}
\newcommand{\vvv}{\mathord{\scaleobj{1.2}{\scalerel*{\usebox{\vvvbox}}{x}}}}
\newcommand{\pivvv}{\pi_{\vvv}}

\newsavebox\evbox
\savebox\evbox{\tikz{
		\draw[black,fill=black] (90:1) circle (.35);
		\draw[black,fill=black] (210:1) circle (.35);
		\draw[black,fill=black] (330:1) circle (.35);
		\draw[black,line width=0.28cm ] (210:1) -- (330:1);
		\draw[opacity=0] (0:1.2) circle (0.1);
	}}
\newcommand{\ev}{\mathord{\scaleobj{1.2}{\scalerel*{\usebox{\evbox}}{x}}}}	
\newcommand{\piev}{\pi_{\ev}}

\newsavebox\eebox
\savebox\eebox{\tikz{
		\draw[black,fill=black] (90:1) circle (.35);
		\draw[black,fill=black] (210:1) circle (.35);
		\draw[black,fill=black] (330:1) circle (.35);
		\draw[black,line width=0.28cm ] (90:1) -- (330:1);
		\draw[black,line width=0.28cm ] (90:1) -- (210:1);
		\draw[opacity=0] (0:1.2) circle (0.1);
	}}
\newcommand{\ee}{\mathord{\scaleobj{1.2}{\scalerel*{\usebox{\eebox}}{x}}}}
\newcommand{\piee}{\pi_{\ee}}

\newsavebox\eeebox
\savebox\eeebox{\tikz{
		\draw[black,fill=black] (90:1) circle (.35);
		\draw[black,fill=black] (210:1) circle (.35);
		\draw[black,fill=black] (330:1) circle (.35);
		\draw[black,line width=0.28cm ] (90:1) -- (330:1);
		\draw[black,line width=0.28cm ] (90:1) -- (210:1);
		\draw[black,line width=0.28cm ] (210:1) -- (330:1);
		\draw[opacity=0] (0:1.2) circle (0.1);
	}}

\theoremstyle{plain}
\newtheorem{thm}{Theorem}[section]
\newtheorem{fact}[thm]{Fact}
\newtheorem{prop}[thm]{Proposition}

\newtheorem{lemma}[thm]{Lemma}

\theoremstyle{definition}
\newtheorem{rem}[thm]{Remark}
\newtheorem{dfn}[thm]{Definition}

\let\eps=\varepsilon
\let\theta=\vartheta
\let\rho=\varrho
\let\phi=\varphi

\def\NN{\mathbb N}
 
\def\PP{\mathbb P}

\def\cA{{\mathcal A}}

\def\cM{{\mathcal M}}
\def\cF{{\mathcal F}}

\def\cC{{\mathcal C}}
\def\cD{{\mathcal D}}
\def\cP{{\mathcal P}}

\def\cE{{\mathcal E}}

\def\cK{{\mathcal K}}

\def\ccC{{\mathscr{C}}}

\def\ccP{{\mathscr{P}}}

\def\ccS{\mathscr{S}}

\DeclareMathOperator{\ex}{ex}

\let\polishlcross=\l
\def\l{\ifmmode\ell\else\polishlcross\fi}

\makeatletter
\def\moverlay{\mathpalette\mov@rlay}
\def\mov@rlay#1#2{\leavevmode\vtop{%
   \baselineskip\z@skip \lineskiplimit-\maxdimen
   \ialign{\hfil$\m@th#1##$\hfil\cr#2\crcr}}}
\newcommand{\charfusion}[3][\mathord]{
    #1{\ifx#1\mathop\vphantom{#2}\fi
        \mathpalette\mov@rlay{#2\cr#3}
      }
    \ifx#1\mathop\expandafter\displaylimits\fi}
\makeatother

\newcommand{\dcup}{\charfusion[\mathbin]{\cup}{\cdot}}
\newcommand{\bigdcup}{\charfusion[\mathop]{\bigcup}{\cdot}}

\def\tand{\ \text{and}\ }
\def\qand{\quad\text{and}\quad}

\let\st=\colon
\def\bl{\bigl(}
\def\br{\bigr)}

\let\setminus=\smallsetminus
\let\emptyset=\varnothing
\let\vn=\varnothing

\begin{document}
\title[Some remarks on $\piee$]{Some remarks on $\piee$}

\author[Christian Reiher]{Christian Reiher}
\address{Fachbereich Mathematik, Universit\"at Hamburg, Hamburg, Germany}
\email{Christian.Reiher@uni-hamburg.de}

\author[Vojt\v{e}ch R\"{o}dl]{Vojt\v{e}ch R\"{o}dl}
\address{Department of Mathematics and Computer Science, 
Emory University, Atlanta, USA}
\email{rodl@mathcs.emory.edu}
\thanks{The second author was supported by NSF grants DMS 1301698 and 1102086.}

\author[Mathias Schacht]{Mathias Schacht}
\address{Fachbereich Mathematik, Universit\"at Hamburg, Hamburg, Germany}
\email{schacht@math.uni-hamburg.de}
\thanks{The third author was supported through the Heisenberg-Programme of the DFG}


\dedicatory{Dedicated to Ron Graham on the occasion of his $80$th birthday}

\begin{abstract}
We investigate extremal problems for hypergraphs satisfying the following density condition. 
A $3$-uniform hypergraph $H=(V, E)$ is \emph{$(d, \eta,\ee)$-dense}
if for any two subsets of pairs $P$, $Q\subseteq V\times V$
the number of pairs $((x,y),(x,z))\in P\times Q$
with $\{x,y,z\}\in E$  is at least $d\,|\cK_{\ee}(P,Q)|-\eta\,|V|^3,$
where $\cK_{\ee}(P,Q)$ denotes the set of pairs in $P\times Q$ of the form $((x,y),(x,z))$. 
For a given $3$-uniform hypergraph $F$ we are interested in the infimum $d\geq 0$ 
such that for sufficiently small $\eta$ every sufficiently large $(d, \eta,\ee)$-dense
hypergraph~$H$ contains a copy of $F$ and this infimum will be denoted by $\piee(F)$.
We present a few results for the case when $F=K_k^{(3)}$ is a complete three uniform hypergraph on $k$ vertices.
It will be shown that $\piee(K_{2^r}^{(3)})\leq \frac{r-2}{r-1}$, which is sharp 
for $r=2,3,4$, where the lower bound for $r=4$ is based on
a result of Chung and Graham [\emph{Edge-colored complete graphs with precisely colored subgraphs},
Combinatorica~\textbf{3} (3-4), 315--324].
\end{abstract} 

\maketitle

\section{Introduction}  
\label{sec:intro}
\subsection{Extremal problems for uniformly dense hypergraphs}
We study extremal problems for $3$-uniform hypergraphs and if not stated otherwise by a hypergraph we always mean 
a $3$-uniform hypergraph. Recall that for a given 
$3$-uniform hypergraph~$F$ the \emph{extremal number $\ex(n, F)$} denotes the maximal number 
of hyperedges a hypergraph $H=(V,E)$ on~$n$ vertices can have without containing a copy of $F$.
Since the sequence 
$
\ex(n, F)/\binom{n}{3}
$
is decreasing, the \emph{Tur\'an density}
\[
\pi(F)=\lim_{n\to\infty}  \frac{\ex(n, F)}{\binom{n}{3}}
\]
is well defined. The study of these extremal parameters was already initiated 
by Tur\'an~\cite{Tu41} more than 70 years ago, but despite a lot of effort still
only very few results are known and several variations were considered.
In particular, Erd\H os and S\'os (see, e.g.,~\cites{ErSo82,Er90}) suggested a variant, where one restricts 
to $F$-free hypergraphs $H$, that are \emph{uniformly dense} on large subsets 
of the vertices. For reals $d\in[0,1]$ and $\eta>0$ we say a $3$-uniform hypergraph 
$H=(V,E)$ is \emph{$(d,\eta,\vvv)$-dense}, if all subsets $X$, $Y$,
$Z\subseteq V$ induce at least $d|X||Y||Z|-\eta|V|^3$ triples $(x,y,z)\in X\times Y\times Z$ 
such that $\{x,y,z\}$ is a hyperedge of $H$. Restricting to \mbox{$\vvv\,$-dense} hypergraphs, the appropriate 
Tur\'an density $\pivvv(F)$ for a given hypergraph $F$ can be defined as
\begin{multline*}
	\pivvv(F)=\sup\bigl\{d\in[0,1]\colon \text{for every $\eta>0$ and $n\in \NN$ there exists}\\
		\text{a $3$-uniform, $F$-free,  $(d,\eta, \vvv)$-dense hypergraph $H$ with $|V(H)|\geq n$}\bigr\}\,.
\end{multline*}
For $F=K_4^{(3)-}$, i.e., the $3$-uniform hypergraph with three hyperedges on four vertices, it was shown by 
Glebov, Kr{\'a}{\soft{l}}, and Volec~\cite{GKV} that $\pivvv(K_4^{(3)-})=1/4$ (see also~\cite{RRS-a}).
However, when $F=K_4^{(3)}$ is the clique on four vertices, the interesting 
conjecture asking if 
\[\pivvv(K_4^{(3)})=\frac{1}{2}\] 
remains still open.

In~\cite{RRS-b} we considered the following stronger density notion. We say a $3$-uniform hypergraph 
$H=(V,E)$ is \emph{$(d,\eta,\ev)$-dense}, if all subsets $X\subseteq V$ 
and sets of pairs $P\subseteq V\times V$ induce at least $d|X||P|-\eta|V|^3$ pairs $(x,(y,z))\in X\times P$ 
such that $\{x,y,z\}$ is a hyperedge of~$H$. Similarly as above, for this concept one defines the 
Tur\'an density $\piev(F)$ for a given hypergraph $F$ by
\begin{multline*}
	\piev(F)=\sup\bigl\{d\in[0,1]\colon \text{for every $\eta>0$ and $n\in \NN$ there exists}\\
		\text{a $3$-uniform, $F$-free,  $(d,\eta, \ev)$-dense hypergraph $H$ with $|V(H)|\geq n$}\bigr\}
\end{multline*}
and the main result in~\cite{RRS-b} asserts $\piev(K_4^{(3)})=1/2$. 
Strictly speaking, in~\cites{RRS-a,RRS-b} \emph{quasirandom} hypergraphs $H$ were considered, i.e., a matching  upper bound
on the number of hyperedges induced on $X\times Y\times Z$ (resp.\ $X\times P$) was formally required. 
However, in the proofs only the density condition (i.e., the lower bound on the number of induced hyperedges) is utilised.
Moreover, for extremal problems involving embeddings of fixed hypergraphs it seems natural to 
restrict only to a lower bound on the density
and here this path will be followed. 

\subsection{Main result}
Here we investigate the following density condition, which further strengthens the notion of 
$\ev\,$-dense hypergraphs and is in some sense the strongest non-trivial density 
condition for extremal problems in $3$-uniform hypergraphs (see, e.g.,~\cite{RRS-b} for a 
more detailed discussion of the different conditions).
\begin{dfn}
\label{eedense}
A $3$-uniform hypergraph $H=(V, E)$ on $n=|V|$ vertices is \emph{$(d, \eta,\ee)$-dense}
if for any two subsets of pairs $P$, $Q\subseteq V\times V$
the number $e_{\ee}(P,Q)$ of pairs of pairs $((x,y),(x,z))\in P\times Q$
with $\{x,y,z\}\in E$  satisfies
\begin{equation}\label{eq:defee}
	e_{\ee}(P,Q)\geq d\,|\cK_{\ee}(P,Q)| - \eta\,n^3\,,
\end{equation}
where $\cK_{\ee}(P,Q)$ denotes the set of pairs in $P\times Q$ of the form $((x,y),(x,z))$. 
The class of $(d, \eta,\ee)$-dense $3$-uniform hypergraphs will be denoted by $\cD(d,\eta,\ee)$.
\end{dfn}

The corresponding Tur\'an density for a given $3$-uniform hypergraph $F$ is then defined by
\begin{multline*}
	\piee(F)=\sup\bigl\{d\in[0,1]\colon \text{for every $\eta>0$ and $n\in \NN$ there exists}\\
		\text{a $3$-uniform, $F$-free  hypergraph $H\in\cD(d,\eta, \ee)$ with $|V(H)|\geq n$}\bigr\}\,.
\end{multline*}

Our main result establishes an upper bound on $\piee(F)$, when $F$ is a complete $3$-uniform hypergraph. 
 
\begin{thm}\label{thm:K2r}
	For every integer $r\geq 2$ and $\eps>0$ there exists an $\eta>0$ and an integer~$n_0$ such that every $3$-uniform $(\frac{r-2}{r-1}+\eps,\eta,\ee)$-dense 
	hypergraph $H$ with at least $n_0$ vertices contains a copy of the complete $3$-uniform hypergraph 
	on $2^r$ vertices $K_{2^r}^{(3)}$, i.e., we have 
	\[
		\piee(K_{2^r}^{(3)})\leq\frac{r-2}{r-1}\,.
	\]
\end{thm}
In particular, for $r=2$ the theorem yields $\piee(K_4^{(3)})=0$ and below we discuss 
lower bound constructions, which show that Theorem~\ref{thm:K2r} is sharp for $r=3$ and $4$ as well. 
For a summary of our results for small cliques 
see~\eqref{eq:results} below. Moreover, a simple general construction given in the next section 
yields $\lim_{k\to \infty}\piee(K_k^{(3)})=1$. 

The upper bound in Theorem~\ref{thm:K2r} shows that the Tur\'an densities for cliques for $\ee\,$-dense hypergraphs 
grow much slower than those for $\ev\,$-dense hypergraphs. In fact, combined with a lower bound 
construction from~\cite{RRS-b} we have 
\[
	 \piee(K_{2^r}^{(3)}) \leq \frac{r-2}{r-1} \leq \piev(K_{r+1}^{(3)})\,.
\]

We also remark that the proof of Theorem~\ref{thm:K2r} extends to $k$-colourable hypergraphs.
Recall that a hypergraph $F$ is \emph{$k$-colourable}, if there exists a partition $V(F)=V_1\dcup\dots\dcup V_k$
such that no hyperedge of $F$ is contained in some $V_i$ for $i\in[k]$. For example, splitting 
the vertex set of $K^{(3)}_{2^r}$ into $2^{r-1}$ sets of size two shows that $K^{(3)}_{2^r}$ is $2^{r-1}$-colourable and the proof of Theorem~\ref{thm:K2r} presented here yields the same upper bound on $\piee(F)$ for any $2^{r-1}$-colourable hypergraph, i.e., if $F$ is $2^{r-1}$-colourable  for some $r\geq 2$ then
\[
	\piee(F)\leq \frac{r-2}{r-1}
\]
(see also Remark~\ref{rem:colour}).

\subsection{Lower bound constructions} In this section we consider constructions yielding lower bounds  for 
$\piee(K_k^{(3)})$ where $k=5$, $6$, and $11$. All constructions given here 
will be probabilistic, which will ensure the required $\ee\,$-denseness (see Proposition~\ref{prop:construct} below). 
For the exclusion of the cliques of given order we shall utilise Ramsey-type arguments.

The following terminology will be useful. For a given finite set $\cC$ of colours, by a \emph{colour pattern} we mean a multiset
containing three (counting with repetition) elements from $\cC$ and a set $\ccP$ of such patterns is called a
\emph{palette} over $\cC$. For example,
\begin{equation}
\label{eq:CGp}
	\ccP=\big\{\{1,1,2\}\,,\ \{1,1,3\}\,,\ \{2,2,3\}\,,\ \{2,2,1\}\,,\ \{3,3,1\}\,,\ \{3,3,2\} \big\}
\end{equation}
is a palette over $\cC=\{1,2,3\}$, which consists of all patterns using exactly two colours. 

For $d\in[0,1]$ we say a palette $\ccP$ over $\cC$ is 
\emph{$(d,\ee)$-dense}, if any pair of  (not necessarily distinct) colours of $\cC$ appears 
in at least $d\,|\cC|$ patterns of $\ccP$. For example, it is easy to check that 
the palette given  in~\eqref{eq:CGp} is $(2/3,\ee)$-dense. 

Next we describe the connection between $(d,\ee)$-dense palettes and $(d,\eta,\ee)$-dense hypergraphs.
For a vertex set $V$ and a colouring $\phi\colon V^{(2)}\to \cC$ of the (unordered) 
pairs of $V$ let $H^\ccP_\phi=(V,E)$ be the $3$-uniform hypergraph defined by 
\[
	E=\big\{\{x,y,z\}\in V^{(3)}\colon \{\phi(x,y),\phi(x,z),\phi(y,z)\}\in\ccP\big\}\,, 
\]
where $\{\phi(x,y),\phi(x,z),\phi(y,z)\}$ is regarded as a multiset. Considering random colourings~$\phi$ 
and a $(d,\ee)$-dense palette $\ccP$ results for any given $\eta>0$ 
with high probability for a sufficiently large set $V$ in a $(d,\eta,\ee)$-dense 
hypergraph $H^\ccP_\phi=(V,E)$.

\begin{prop}\label{prop:construct}
	Suppose $\ccP$ is a $(d,\ee)$-dense palette over some finite set of colours $\cC$. 
	For every $\eta>0$ and all sufficiently large sets $V$ there exists 
	a colouring $\phi\colon V^{(2)}\to\cC$ of the unordered pairs of~$V$ such that 
	the $3$-uniform hypergraph $H^{\ccP}_\phi=(V,E)$ 
	is $(d,\eta,\ee)$-dense.
	
	Moreover, if any colouring of the edges of the complete graph $K_k$ with colours from~$\cC$
	yields a triangle with a pattern not from $\ccP$, then $H^{\ccP}_\phi$ contains no copy of 
	$K_k^{(3)}$ and, consequently, $\piee(K_k^{(3)})\geq d$.
\end{prop}

Before we prove the proposition we deduce some lower bounds 
for $\piee(K_5^{(3)})$, $\piee(K_{6}^{(3)})$, and~$\piee(K_{11}^{(3)})$ from it.
\begin{itemize}
\item We first show $\piee(K_5^{(3)})\geq 1/3$.
For that we consider the $(1/3,\ee)$-dense palette 
	\[
		\ccP=\big\{\{1,1,2\}\,,\ \{2,2,3\}\,,\ \{3,3,1\}\big\}
	\]
	over $\{1,2,3\}$ and we have to show that the edges of the complete graph $K_5$
	cannot be coloured in such a way that all triangles get a pattern from 
	$\ccP$. We consider three cases. If there is a vertex of $K_5$ incident with edges of each colour,
	those three neighbours would have to span a triangle using all three colours, which is 
	not a pattern in $\ccP$. Similarly, if there is a vertex incident with three edges of the same colour, then 
	the triangle in that neighbourhood is monochromatic, which is also not contained in $\ccP$. Hence, every vertex has 
	degree two or zero in the three monochromatic subgraphs of $K_5$ given by the edge colouring. Therefore, the 
	colouring induces a decomposition of $K_5$ into monochromatic cycles. But since monochromatic triangles 
	are not allowed, this decomposition must consist of two cycles of length five. However, this 
	leads to two triangles with the patterns $\{a,a,b\}$ and $\{b,b,a\}$ for some $a\neq b$, but only one 
	of these patterns is in the palette.	
\item Next we verify $\piee(K_6^{(3)})\geq 1/2$.
Indeed, since every two-colouring 
	of the edges of~$K_6$ yields a monochromatic triangle, the assertion follows from the
	$(1/2,\ee)$-dense palette $\{\{1,1,2\}\,,\ \{2,2,1\}\}$
	over $\{1,2\}$. Combining this lower bound with the obvious monotonicity and with 
	Theorem~\ref{thm:K2r} for $r=3$ leads to
	\[
		\frac{1}{2}\leq \piee(K_6^{(3)}) \leq \piee(K_7^{(3)}) \leq \piee(K_8^{(3)}) \leq \frac{1}{2}\,.
	\]
	
	We also remark that this construction can be generalised for multicolour Ramsey numbers.
	For every integer $\l$ let $k=R(3;\l)$ be the $\l$-colour Ramsey number for the triangle. 
	Then the palette consisting of all but the monochromatic patterns over $[\l]=\{1,\dots,\l\}$ yields
	$\piee(K_k^{(3)})\geq (\l-1)/\l$, which shows that
	$\piee(K_k^{(3)})\to1$	
	as~$k$ tends to infinity. 
\item The last construction establishes $\piee(K_{11}^{(3)})\geq 2/3$.
For this we appeal to the palette given in~\eqref{eq:CGp}, which is 
$(2/3,\ee)$-dense over $\{1,2,3\}$. It follows from the result of Chung and Graham from~\cite{CG83} that
every three-colouring of the edges of $K_{11}$ yields either a monochromatic triangle or a rainbow triangle. Since these 
patterns are not in the palette in~\eqref{eq:CGp}, it follows that $\piee(K_{11}^{(3)})\geq 2/3$. Together with 
the case $r=4$ of Theorem~\ref{thm:K2r} this yields
	\[
		\frac{2}{3}\leq \piee(K_{11}^{(3)}) \leq \dots \leq \piee(K_{16}^{(3)}) \leq \frac{2}{3}\,.
	\]
\end{itemize}

Summarising the discussion above for cliques of small size we established
\begin{eqnarray}
	\piee(K_4^{(3)}) = &0&	\nonumber\\
	&\frac{1}{3}&\leq \piee(K_5^{(3)})	\nonumber\\
	\piee(K_6^{(3)})
	=
	\piee(K_7^{(3)})
	=
	\piee(K_8^{(3)})
	=
	&\frac{1}{2}&
	\leq
	\piee(K_9^{(3)})
	\leq 
	\piee(K_{10}^{(3)})	\label{eq:results}\\
	\piee(K_{11}^{(3)})
	=
	\dots
	=
	\piee(K_{16}^{(3)})
	=
	&\frac{2}{3}
	\nonumber
\end{eqnarray}
which leaves gaps for the values of $\piee(K_k^{(3)})$ for $k=5$, $9$, $10$ and it would be very interesting to 
close these.
We conclude this introduction with the short proof of Proposition~\ref{prop:construct}.
\begin{proof}[Proof of Proposition~\ref{prop:construct}]
	For a given $(d,\ee)$-dense palette $\ccP$ over some finite set of colours~$\cC$ and 
	$\eta>0$ we consider a random colouring $\phi\colon V^{(2)}\to\cC$ of the unordered pairs of
	some sufficiently large set~$V$, where each pair is coloured independently and uniformly 
	with one of the colours from~$\cC$. We shall show that with probability tending to one as $|V|\to\infty$
	the hypergraph~$H^{\ccP}_{\phi}$ is $(d,\eta,\ee)$-dense.
	
	We begin with the following observation. For any two subsets $Y$, $Z\subseteq V$
	and any selection $\cC'\subseteq\cC$ of at least $d\,|\cC|$ colours we expect at least 
	$d(|Y||Z|-|Y\cap Z|)$ pairs $(y,z)\in Y\times Z$ with $y\neq z$ such that 
	$\phi(y,z)\in\cC'$. Chernoff's inequality in the form
	\[
		\PP(X<\mathbb{E}X-t)\leq \exp(-\tfrac{t^2}{2\mathbb{E}X})
	\]
	applied with $t=\eta |V|^2/|\cC|^2-|Y\cap Z|$ shows that at least
	\begin{equation}\label{eq:c1}
		d|Y||Z|-\eta|V|^2/|\cC|^2
	\end{equation}
	such pairs will be present with probability 
	at least $1-\exp(-\tfrac{\eta^2|V|^2}{3|\cC|^4})$. Consequently, applying the union bound 
	over all choices of $Y$, $Z\subseteq V$ and $\cC'\subseteq \cC$ we infer that 
	with probability tending to 1 as $|V|\to\infty$ that the bound in~\eqref{eq:c1} 
	holds for all these 
	choices and for the rest of the proof we assume that $\phi$ satisfies this 
	property.
	
	Let $P$, $Q\subseteq V\times V$. For a colour $c\in\cC$ and a vertex $x\in V$ set 
	\[
		N^c_{P}(x)=\bigl\{y\in V\colon (x,y)\in P\tand \phi(x,y)=c\bigr\}
	\] 
	and, similarly, define $N^c_{Q}(x)$. Clearly, we have 
	\[
		|\cK_{\ee}(P,Q)|=\sum_{x\in V}\sum_{c,c'\in\cC}|N^c_{P}(x)||N^{c'}_{Q}(x)|\,.
	\]
	Since the palette $\ccP$ is $(d,\ee)$-dense, for any (not necessarily distinct) colours 
	$c$, $c'\in\cC$ the set $\cC'=\{c''\in\cC\colon \{c,c',c''\}\in\ccP\}$
	has size at least $d|\cC|$. Hence, using the lower bound given in~\eqref{eq:c1}
	for $Y=N^c_{P}(x)$ and $Z=N^{c'}_{Q}(x)$
	yields that there are at least
	\[
		d|N^c_{P}(x)||N^{c'}_{Q}(x)|-\eta|V|^2/|\cC|^2
	\]
	triples $(x,y,z)$ with $(x,y)\in P$, $\phi(x,y)=c$, $(x,z)\in Q$, $\phi(x,z)=c'$, and $\{x,y,z\}\in E(H_\phi^{\ccP})$.
	Summing this estimate over all vertices $x\in V$ and colours $c$, $c'\in\cC$
	leads to 
	\begin{align*}
		e_{\ee}(P,Q)&\geq \sum_{x\in V}\sum_{c,c'\in\cC}\Big(d|N^c_{P}(x)||N^{c'}_{Q}(x)|- \eta|V|^2/|\cC|^2\Big) \\
		&=d\sum_{x\in V}\sum_{c,c'\in\cC}|N^c_{P}(x)||N^{c'}_{Q}(x)|-\eta|V|^3\\
		&=d|\cK_{\ee}(P,Q)|-\eta|V|^3\,,
	\end{align*}
	which shows that $H_\phi^\ccP$ is $(d,\eta,\ee)$-dense.
		
	The moreover-part follows directly from the assumed Ramsey-type property of $K_k$ and the 
	definitions of $H^{\ccP}_{\phi}$ and $\piee(\cdot)$. 
\end{proof}

\section{Hypergraph regularity method}\label{sec:regmethod}
A key tool in the proof of Theorem~\ref{thm:K2r} is the regularity lemma for $3$-uniform hypergraphs. 
We follow the approach from~\cites{RoSchRL,RoSchCL} combined with the results from~\cite{Gow06} and~\cite{NPRS09}
and below we introduce the necessary notation.

For two disjoint sets $X$ and $Y$ we denote by $K(X,Y)$ the complete bipartite graph with that vertex partition.
We say a bipartite graph $P=(X\dcup Y,E)$ is \emph{$(\delta_2, d_2)$-regular} if for all subsets 
$X'\subseteq X$ and $Y'\subseteq Y$ we have 
\[
	\big|e(X',Y')-d_2|X'||Y'|\big|\leq \delta_2 |X||Y|\,,
\]
where $e(X',Y')$ denotes the number of edges of $P$ with one vertex in $X'$ and one vertex in~$Y'$.
Moreover, for $k\geq 2$ we say a $k$-partite graph $P=(X_1\dcup \dots\dcup X_k,E)$ is $(\delta_2, d_2)$-regular, 
if all of its  $\binom{k}{2}$ naturally 
induced bipartite subgraphs $P[X_i,X_j]$ are $(\delta_2, d_2)$-regular. 
For a tripartite graph $P=(X\dcup Y\dcup Z,E)$
we denote by $\cK_3(P)$ the triples of vertices spanning a triangle in~$P$, i.e., 
\[
	\cK_3(P)=\{\{x,y,z\}\subseteq X\cup Y\cup Z\colon xy, xz, yz\in E\}\,.
\]
If the tripartite graph $P$ is $(\delta_2, d_2)$-regular, then the so-called \emph{triangle counting lemma}
implies
\begin{equation}
	\label{eq:TCL}
		|\cK_3(P)|\leq d_2^3|X||Y||Z|+3\delta_2|X||Y||Z|\,.
\end{equation}

We say a $3$-uniform hypergraph $H=(V,E_H)$ is regular with respect to a tripartite graph~$P$ if it matches 
approximately
the same proportion of triangles for every subgraph $Q\subseteq P$. This we make precise in the following definition.

\begin{dfn}
\label{def:reg}
A $3$-uniform hypergraph $H=(V,E_H)$ is \emph{$(\delta_3,d_3)$-regular with respect to 
a tripartite graph $P=(X\dcup Y\dcup Z,E_P)$} 
with $V\supseteq  X\cup Y\cup Z$ if for every tripartite subgraph $Q\subseteq P$ we have 
\[
	\big||E_H\cap\cK_3(Q)|-d_3|\cK_3(Q)|\big|\leq \delta_3|\cK_3(P)|\,.
\]
Moreover, we simply say \emph{$H$ is $\delta_3$-regular with respect to $P$}, if it is $(\delta_3,d_3)$-regular for some $d_3\geq 0$.
We also define the \emph{relative density} of $H$ with respect to $P$ by
\[
	d(H|P)=\frac{|E_H\cap\cK_3(P)|}{|\cK_3(P)|}\,,
\]
where we use the convention $d(H|P)=0$ if $\cK_3(P)=\emptyset$. If $H$ is not $\delta_3$-regular with respect to $P$, then we simply refer to it as \emph{$\delta_3$-irregular}.
\end{dfn}

The regularity lemma for $3$-uniform hypergraphs, introduced by Frankl and R\"odl in~\cite{FR}, provides for 
every hypergraph $H$ a partition of its vertex set and a partition of the edge sets of the complete bipartite 
graphs induced by the vertex partition such that for appropriate constants $\delta_3$, $\delta_2$, and $d_2$ 
\begin{enumerate}[label=\nlabel]
	\item the bipartite graphs given by the partitions are $(\delta_2,d_2)$-regular and
	\item $H$ is $\delta_3$-regular with respect to ``most'' tripartite graphs $P$ given by the partition.
\end{enumerate}
In many proofs based on the regularity method it is
convenient to ``clean'' the regular partition provided by the regularity lemma. In particular, 
we shall disregard hyperedges of~$H$ that belong to $\cK_3(P)$ when $H$ is not $\delta_3$-regular or 
when $d(H|P)$ is very small. These properties are rendered in the following somewhat standard
corollary of the regularity lemma.

\begin{thm}
	\label{thm:TuRL}
	For every $d_3>0$, $\delta_3>0$ and  $m\in\NN$, and every function $\delta_2\colon \NN \to (0,1]$,
	there exist integers~$T_0$ and $n_0$ such that for every $n\geq n_0$
	and every $n$-vertex $3$-uniform hypergraph $H=(V,E)$ the following holds.
	
	There exists a subhypergraph $\hat H=(\hat V,\hat E)\subseteq H$, an integer $\l\leq T_0$,
	a vertex partition $V_1\dcup\dots\dcup V_m=\hat V$, 
	and for all $1\leq i<j\leq m$ there exists 
	a partition 
	\[
		\cP^{ij}=\{P^{ij}_\alpha=(V_i\dcup V_j,E^{ij}_\alpha)\colon 1\leq \alpha \leq \l\}
	\] 
	of $K(V_i,V_j)$ satisfying the following properties
	\begin{enumerate}[label=\rmlabel]
		\item\label{TuRL:1} $|V_1|=\dots=|V_m|\geq (1-\delta_3)n/T_0$,
		\item\label{TuRL:2} for every $1\leq i<j\leq m$ and $\alpha\in [\l]$ the bipartite graph $P^{ij}_\alpha$ is $(\delta_2(\l),1/\l)$-regular,
		\item $\hat H$ is $\delta_3$-regular with respect to $P^{ijk}_{\alpha\beta\gamma}$
			for all tripartite graphs (which will be later referred to as triads)
			\begin{equation}\label{eq:triad}
				P^{ijk}_{\alpha\beta\gamma}=P^{ij}_\alpha\dcup P^{ik}_\beta\dcup P^{jk}_\gamma=(V_i\dcup V_j\dcup V_k, E^{ij}_\alpha\dcup E^{ik}_{\beta}\dcup E^{jk}_{\gamma})\,,
			\end{equation}
			with $1\leq i<j<k\leq m$ and $\alpha$, $\beta$, $\gamma\in[\l]$, where the density $d(\hat H|P^{ijk}_{\alpha\beta\gamma})$ is 
			either $0$ or at least $d_3$, and			
		\item\label{TuRL:4} for every $1\leq i<j<k\leq m$ 												there are at most $\delta_3\,\ell^3$ triples $(\alpha, \beta, \gamma)\in[\ell]^3$
			such that $d(\hat H|P^{ijk}_{\alpha\beta\gamma})<d(H|P^{ijk}_{\alpha\beta\gamma})-d_3$.
	\end{enumerate}
\end{thm}
The standard proof of Theorem~\ref{thm:TuRL} based on a refined version of 
the regularity lemma from~\cite{RoSchRL}*{Theorem~2.3}
can be found in~\cite{RRS-a}*{Corollary 3.3}. Actually the statement there differs from the one given here in the final clause, but the proof from~\cite{RRS-a} shows the present version as well. In fact, the new version of \ref{TuRL:4} is a consequence of clause (a) in the definition of the hypergraph $R$ in~\cite{RRS-a}*{Proof of Corollary 3.3}, as we only remove more than~$d_3|\cK(\cP^{ijk}_{\alpha\beta\gamma})|$ hyperedges from $H$ to obtain $\hat H$, when $H$ is $\delta_3$-irregular with respect to~$\cP^{ijk}_{\alpha\beta\gamma}$.

We shall use a so-called \emph{counting/embedding lemma}, which allows us to embed hypergraphs of fixed isomorphism type 
into appropriate and sufficiently regular and dense triads of the partition provided by Theorem~\ref{thm:TuRL}. 
The following statement is a direct consequence of~\cite{NPRS09}*{Corollary~2.3}.

\begin{thm}[Embedding Lemma]
	\label{thm:EL}
	For every $3$-uniform hypergraph $F=(V_F,E_F)$ with vertex set $V_F=[f]$
	and every $d_3>0$ there exists $\delta_3>0$, and functions 
	$\delta_2\colon \NN\to(0,1]$ and  $N\colon \NN\to\NN$
	such that the following holds for every $\l\in\NN$.

	Suppose $P=(V_1\dcup\dots\dcup V_f, E_P)$ is a $(\delta_2(\l),\frac{1}{\l})$-regular, $f$-partite graph
	with vertex classes satisfying  $|V_1|=\dots=|V_f|\geq N(\l)$ and suppose $H$ is an $f$-partite, $3$-uniform hypergraph
	such that for every edge $ijk\in E_F$ we have
	\begin{enumerate}[label=\alabel]
		\item\label{EL:a} $H$ is $\delta_3$-regular with respect to to the tripartite subgraph $P[V_i\dcup V_j\dcup V_k]$ and 
		\item\label{EL:b} $d(H|P[V_i\dcup V_j\dcup V_k])\geq d_3$
	\end{enumerate} 
	then $H$ contains a copy of $F$, where for every $i\in [f]=V_F$ the image of $i$ is contained in~$V_i$.
\end{thm}
In an application of Theorem~\ref{thm:EL} the tripartite graphs $P[V_i\dcup V_j\dcup V_k]$ in~\ref{EL:a} and~\ref{EL:b}
will be given by triads  $P^{ijk}_{\alpha\beta\gamma}$ from the partition given by Theorem~\ref{thm:TuRL}.

For the proof of Theorem~\ref{thm:K2r} we consider a $\ee\,$-dense hypergraph~$H$ of density $\frac{r-2}{r-1}+\eps$.
We will apply the regularity lemma in the form of Theorem~\ref{thm:TuRL} to~$H$. The main part of the proof
concerns the appropriate selection of dense and regular triads, that are ready for an application of the embedding lemma
with $F=K^{(3)}_{2^r}$. This will be the focus in Sections~\ref{sec:K2r} and~\ref{sec:embred}.

\section{Reduced hypergraphs}
\label{sec:K2r}

Like many other proofs based on the regularity method, the proof of Theorem~\ref{thm:K2r} 
will factor naturally
through an auxiliary statement speaking about certain ``reduced hypergraphs'' that we would 
like to describe next. 

Consider any finite set of indices $I$, suppose that associated with any two distinct 
indices $i, j\in I$ we have a finite nonempty set of vertices $\cP^{ij}=\cP^{ji}$, and that for 
distinct pairs of indices the corresponding vertex classes are disjoint. Assume further 
that for any three distinct indices $i, j,k\in I$ we are given a tripartite $3$-uniform 
hypergraph $\cA^{ijk}$ with vertex classes $\cP^{ij}$, $\cP^{ik}$, and $\cP^{jk}$. Under 
such circumstances we call the $\binom{|I|}{2}$-partite $3$-uniform hypergraph $\cA$ 
defined by
\[
V(\cA)=\bigdcup_{\{i,j\}\in I^{(2)}} \cP^{ij}
\qquad \text{ and } \qquad
E(\cA)=\bigdcup_{\{i,j, k\}\in I^{(3)}} E(\cA^{ijk})
\]
a {\it reduced hypergraph}. We also refer to $I$ as the {\it index set} of $\cA$, to the sets 
$\cP^{ij}$ as the {\it vertex classes} of $\cA$, and to the hypergraphs $\cA^{ijk}$ as the 
{\it constituents} of $\cA$. 

This concept of a reduced hypergraph might look a bit artificial at first, especially since only
$\binom{|I|}{3}$ out of the $\binom{\binom{|I|}{2}}{3}$ naturally induced tripartite 
subhypergraphs are inhabited. However, as it turns out these reduced hypergraphs 
are well suited for analyzing the structure of the partition provided by
Theorem~\ref{thm:TuRL} applied to a given hypergraph $H$.     

Now, when $H$ happens to be $(d, \eta, \ee)$-dense, then the corresponding reduced hypergraph~$\cA$ inherits a property reflecting this. We are thus led to the notion of 
a reduced hypergraph $\cA$ being $(d, \delta, \ee)$-dense. 
Roughly speaking, this means that all constituents of~$\cA$ are required to satisfy a
$\delta$-approximate pair-degree condition with proportion $d$, which is rendered in the following definition.

\begin{dfn}\label{def:denseA}
A reduced hypergraph $\cA$ with index set $I$ is 
\emph{$(d, \delta, \ee)$-dense} for some $d\in [0,1]$ and $\delta>0$,
if for any three distinct $i, j, k\in I$ the following is true: 

There are at most $\delta\,|\cP^{ij}|\,|\cP^{ik}|$ pairs of vertices 
$(P^{ij}, P^{ik}) \in \cP^{ij} \times \cP^{ik}$
with the property that there are fewer than $d\,|\cP^{jk}|$ vertices $P^{jk}\in \cP^{jk}$ 
for which $\{P^{ij}, P^{ik}, P^{jk}\}\in E(\cA^{ijk})$ holds.
\end{dfn}

For an integer $t\ge 3$ we say that a reduced hypergraph $\cA$ {\it contains a clique of 
order~$t$} if there are
\begin{enumerate}
\item[$\bullet$] 
a set $J\subseteq I$ with $|J|=t$
\item[$\bullet$]
and for any two distinct indices $i,j \in J$ a vertex $P^{ij}\in \cP^{ij}$
\end{enumerate}
such that we have 
$\{P^{ij}, P^{ik}, P^{jk}\}\in E(\cA^{ijk})$ for any three distinct $i,j,k\in J$. 

Now the statement to which we may reduce Theorem~\ref{thm:K2r} via the regularity method
is the following.

\begin{prop}\label{lem:reduced}
Given an integer $r\ge 2$ and a real $\eps>0$, there exists a real $\delta>0$ and an integer $m$
such that every $\bigl(\frac{r-2}{r-1}+\eps, \delta, \ee\bigr)$-dense reduced hypergraph 
whose index set has size at least $m$ contains a clique of order $2^r$.
\end{prop}

\begin{proof}[Proof of Theorem~\ref{thm:K2r} assuming Proposition~\ref{lem:reduced}] 
Roughly speaking, this reduction consists of two parts. 
Given a $\bigl(\frac{r-2}{r-1}+\eps,\eta,\ee\bigr)$-dense hypergraph $H$
we will apply the regularity lemma in the form of  Theorem~\ref{thm:TuRL}
and obtain a reduced hypergraph $\cA$ for $\hat H\subseteq H$.
In the first part we then verify that for an appropriate choice of the involved constants 
the reduced hypergraph~$\cA$ is indeed 
$\bigl(\frac{r-2}{r-1}+\eps/4, \delta, \ee\bigr)$-dense.
This allows for an application of 
Proposition~\ref{lem:reduced} yielding a clique of order $2^r$ in $\cA$. 
In the second part it remains to check that this clique of order~$2^r$ in $\cA$
defines an appropriate collection of triads ready for an application of the embedding lemma (Theorem~\ref{thm:EL})
yielding a copy of~$K^{(3)}_{2^r}$ in $\hat H\subseteq H$. Below we give the details of this proof.

Given $r\geq 2$ and $\eps>0$ we fix auxiliary constants and functions to satisfy the hierarchy
\begin{equation}\label{eq:phier}
	\tfrac{1}{r}\,,\eps\gg \delta \gg \tfrac{1}{m}\,, d_3  \gg \delta_3\gg\tfrac{1}{\l} \gg \delta_2(\l)\,,\tfrac{1}{N(\l)}\gg\tfrac{1}{T_0}\gg \eta
\end{equation}
where $\delta$ and $m$ are given by Proposition~\ref{lem:reduced} applied with $r$ and $\eps/4$;
and $\delta_3$, and the functions $\delta_2(\cdot)$, and $N(\cdot)$ are given by Theorem~\ref{thm:EL} applied for 
$F=K_{2^r}^{(3)}$ and $d_3$; and $T_0$ is given by Theorem~\ref{thm:TuRL}.

For a $(\frac{r-2}{r-1}+\eps,\eta,\ee)$-dense hypergraph $H=(V,E)$ on sufficiently many vertices, we apply 
the regularity lemma in the form of Theorem~\ref{thm:TuRL} and obtain a subhypergraph 
\[
	\hat H=(\hat V, \hat E)\subseteq H\,,
\]
some integer $\l\leq T_0$, a vertex partition $V_1\dcup \dots\dcup V_m=\hat V$, 
and bipartite graphs $P^{ij}_{\alpha}$ for all~$i$ and $j$ with $1\leq i<j\leq m$ and every $\alpha\in[\l]$ satisfying
properties~\ref{TuRL:1}--\ref{TuRL:4} of Theorem~\ref{thm:TuRL}. For the index set $I=[m]$
we consider the naturally given reduced hypergraph $\cA$ for the regular partition of~$\hat H$ 
with vertex classes
\begin{align*}
	\cP^{ij}&=\{P^{ij}_{\alpha}\colon \alpha\in[\l]\}\\
\intertext{for all distinct $i$, $j\in I$ and with constituents $\cA^{ijk}$ for distinct $i$, $j$, $k\in I$ with}
	E(\cA^{ijk})&=\big\{\{P^{ij}_\alpha,P^{ik}_\beta,P^{jk}_\gamma\}\colon 
		(\alpha,\beta,\gamma)\in[\l]^3 \tand d(\hat H|P^{ijk}_{\alpha\beta\gamma})\geq d_3\big\}\,.
\end{align*}
Next we check that the reduced hypergraph $\cA$ is $\bigl(\frac{r-2}{r-1}+\frac{\eps}{4},\,\delta,\ee\bigr)$-dense.
Given distinct indices $i$, $j$, $k\in I$ and $P^{ij}\in\cP^{ij}$ and $P^{ik}\in\cP^{ik}$
it follows from the so-called graph counting lemma for graphs that for the $(\delta_2(\l),1/\l)$-regular bipartite graphs 
$P^{ij}$ and~$P^{ik}$ we have  
\begin{equation}\label{eq:p2}
	\Big|\big|\cK_{\ee}(P^{ij},P^{ik})\big|- \frac{1}{\l^2}|V_i||V_j||V_k|\Big|\leq 2\delta_2(\l)|V_i||V_j||V_k|\,.
\end{equation}
Consequently, the $(\frac{r-2}{r-1}+\eps,\eta,\ee)$-denseness of $H$ 
implies that the number $e^H_{\ee}(P^{ij},P^{ik})$ of hyperedges in $H$ matching $P_2$'s from $\cK_{\ee}(P^{ij},P^{ik})$
satisfies
\begin{align*}
	e^H_{\ee}(P^{ij},P^{ik})
	&\geq
	\left(\frac{r-2}{r-1}+\eps\right)\frac{1}{\l^2}|V_i||V_j||V_k|-2\delta_2(\l)|V_i||V_j||V_k|-\eta n^3\\
	&\geq
	\left(\frac{r-2}{r-1}+\frac{\eps}{2}\right)\frac{1}{\l^2}|V_i||V_j||V_k|\,.
\end{align*}
Owing to~\eqref{eq:TCL} we have $|\cK_3(P^{ij}\cup P^{ik}\cup P^{jk})|\leq |V_i||V_j||V_k|/\l^3 + 3\delta_2(\l)|V_i||V_j||V_k|$ for every 
$P^{jk}\in\cP^{jk}$ and combined with the upper bound in~\eqref{eq:p2} and $\delta_2(\l)\ll 1/\l\ll\eps$ we obtain
\begin{equation}\label{eq:pa}
	\big|\big\{P^{jk}\in\cP^{jk}\colon d(H|P^{ij}\cup P^{ik}\cup P^{jk})\geq d_3\big\}\big|
	\geq 
	\left(\frac{r-2}{r-1}+\frac{\eps}{3}-d_3\right)\l
\end{equation}
for any given pair $(P^{ij},P^{ik})\in\cP^{ij}\times\cP^{ik}$. 

Property~\ref{TuRL:4} of Theorem~\ref{thm:TuRL}
implies that for all but up to at most $\sqrt{\delta_3}\l^2$ pairs $(P^{ij},P^{ik})\in\cP^{ij}\times\cP^{ik}$
there are at most $\sqrt{\delta_3}\l$ graphs $P^{jk}\in\cP^{jk}$ such that  
\[
	d(\hat H|P^{ij}\cup P^{ik}\cup P^{jk})=0\,,\ \text{while}\ d(H|P^{ij}\cup P^{ik}\cup P^{jk})\geq d_3\,.
\] 
Consequently, from~\eqref{eq:pa} it follows that 
for all but at most $\sqrt{\delta_3}\l^2$ pairs $(P^{ij},P^{ik})\in\cP^{ij}\times\cP^{ik}$ there are at least 
\[
	\left(\frac{r-2}{r-1}+\frac{\eps}{3}-d_3-\sqrt{\delta_3}\right)\l
	\geq
	\left(\frac{r-2}{r-1}+\frac{\eps}{4}\right)|\cP^{jk}|
\]
graphs $P^{jk}\in\cP^{jk}$ such that $\{P^{ij},P^{ik},P^{jk}\}\in E(\cA)$. 
In other words, since $\sqrt{\delta_3}\leq \delta$  the reduced hypergraph $\cA$ 
for $\hat H$ is $\bigl(\frac{r-2}{r-1}+\frac{\eps}{4},\delta,\ee\bigr)$-dense.

Proposition~\ref{lem:reduced} then shows that $\cA$ contains a clique of oder $2^r$, i.e., there exists  $J\subseteq [m]$ of 
size $2^r$ and bipartite graphs $P^{ij}\in\cP^{ij}$ for any distinct $i$, $j\in J$ such that $\{P^{ij},P^{ik},P^{jk}\}$
is a hyperdge of $\cA$ for all distinct $i$, $j$, $k\in J$. By the definition of $\cA$ this shows that 
$d(\hat H|P^{ij}\cup P^{ik}\cup P^{jk})\geq d_3$ and, hence, we may apply the embedding lemma (Theorem~\ref{thm:EL}) to 
$\hat H[\bigcup_{j\in J} V_j]$ and $P=\bigcup_{\{i,j\}\in J^{(2)}}P^{ij}$ to obtain the desired clique $K^{(3)}_{2^r}$ in $\hat H\subseteq H$. 
\end{proof}
It is left to verify Proposition~\ref{lem:reduced}, which will be the content of the next section.

\section{Embedding cliques in the reduced hypergraph}
\label{sec:embred}
In this section we shall provide a proof of Proposition~\ref{lem:reduced}.
This will involve several inductions, which will require to prove a more general and somewhat technical
statement (see Proposition~\ref{lem:fortress}). Instead of proving it 
directly it appears preferable to state and prove an even more general Proposition~\ref{lem:general}. 
We will show that
\[
	\text{Proposition~\ref{lem:general} $\Longrightarrow$ Proposition~\ref{lem:fortress}
		$\Longrightarrow$ Proposition~\ref{lem:reduced}\,,}
\]
and thus the
proof of Theorem~\ref{thm:K2r} will be complete with the proof of Propostion~\ref{lem:general}.  

To facilitate the wording of these generalisations, we introduce some further concepts. 
We will frequently deal with finite sequences of the form 
$a=(a_1, \ldots, a_k)$, where~$k$ is a nonnegative integer. If $k=0$, then $a$ is the 
{\it empty sequence} denoted by $\vn$. Generally, $k$ is called the {\it length} of $a$ 
and we express this by writing $k=|a|$. 
For an integer $\ell\in [0,k]$ the {\it restriction} $a|\ell$ is defined to be the
{\it initial segment} $(a_1, \ldots, a_\ell)$ of $a$ and for $\l\in[k]$ we denote the $\l$-th element 
of $a$ by $a(\l)$, i.e., $a(\l)=a_{\l}$. A {\it direct continuation} of~$a$ is a
finite sequence $b$ obtainable from~$a$ by appending an arbitrary 
further term to it, so that~$b$
satisfies $|b|=k+1$ and $b|k=a$. 
The {\it concatenation} of two finite sequences
$a=(a_1, \ldots, a_k)$ and $b=(b_1, \ldots, b_\ell)$ is defined to be 
$(a_1, \ldots, a_k, b_1, \ldots, b_\ell)$ and denoted by~$a\circ b$. 
Moreover, the longest common initial segment of $a$ and $b$ is denoted by 
$a\wedge b$. The following notion of regular trees will be useful.

\begin{dfn}[{$[k,M]$-system}]
For integers $k, M\ge 1$ an {\it $M$-ary tree of height $k$} is
a set~$T$ of finite sequences whose length is at most $k$, such that 
\begin{enumerate}
\item[$\bullet$]
$\vn\in T$, sometimes called the \emph{root} of $T$, and
\item[$\bullet$] 
every $a\in T$ with $|a|<k$ has precisely $M$ direct continuations in $T$. 
\end{enumerate}
The set of elements $\sigma$ that extend some $a=(a_1,\dots,a_\l)\in T$ to a direct continuation $a\circ (\sigma)\in T$
are the \emph{successors of $a$} denoted by 
\[
	\ccS_T(a)=\{\sigma\colon (a_1,\dots,a_\l,\sigma)\in T\}\,.
\]
The {\it leaves} of $T$ are its elements of length $k$ and the set of these leaves is denoted 
by $[T]$. We say a set $S$ is a {\it $[k, M]$-system} if $S=[T]$ holds for some 
$M$-ary tree $T$ of height~$k$. 
\end{dfn}
Giving two examples, we would like to mention that any set consisting of $M$ elements can be viewed as a $[1, M]$-system, whilst the boolean cube $\{0, 1\}^k$ is a $[k, 2]$-system. 
Moreover, we notice that we have $|S|=M^k$ for any $[k,M]$-system $S$. 

In the iterated Ramsey-type 
arguments that we use in the proofs of this section we
will move from $[k,M]$-systems to $[k,m]$-systems for some $m\ll M$. However, while we can 
preserve the tree structure, we have no control about the subtree of the original 
$[k,M]$-system that will be kept after a Ramsey argument. In fact, this is the reason why 
we prefer to work with trees and $[k,M]$-systems, instead of sets $S$ of 
the form $S=\cM^k$ for some $M$-element set~$\cM$.

For the proof of Proposition~\ref{lem:reduced} we are given a 
$\bigl(\frac{r-2}{r-1}+\eps, \delta, \ee\bigr)$-dense reduced hypergraph~$\cA$ with  index set $I$
and we may assume that $\delta\ll |I|^{-1}\ll\eps$. 
We then need to obtain some $J\subseteq I$ with $|J|=2^r$ that spans a clique in $\cA$.
For that we will view $I$ as an $[r, M]$-system for some large integer $M$ and 
use Ramsey-type arguments for shrinking~$I$ down to an appropriate 
$[r, 2]$-system~$J$, such that the required vertices and edges of $\cA$ exist (see Definition~\ref{def:fortress}
and Fact~\ref{fact:fortress} below). 
We begin with the following observation, which follows by a simple averaging argument, and  which 
will be utilised in the proof of Proposition~\ref{lem:general}.

\begin{lemma}\label{lem:system}
Given two integers $k, M\ge 1$, let $S$ be a $[k, M]$-system and let $X$ be a subset of $S$
satisfying $|X|\ge \eps M^k$ for some $\eps>0$. Then for some integer $m\ge \eps M/k$ there
exists a $[k, m]$-system $S'\subseteq X$. 
\end{lemma}

\begin{proof}
We argue by induction on $k$. The base case $k=1$ is clear because $X$ is automatically 
going to be a $[1, |X|]$-system. Now suppose that $k\ge 2$ and that the lemma holds for 
$k-1$ in place of $k$. Let $S$ and $X$ be as above and denote the underlying tree of $S$
by $T$. Clearly the set $A=\{a\in T\st |a|=1\}$ has size $M$ and for every $a\in A$ the set 
$S_a=\{b\st a\circ b\in S\}$ is a $[k-1, M]$-system. Consider for every $a\in A$ the set
\[
X_a=\{b\in S_a\st a\circ b\in X\} \qquad
\] 
and let
\[
A'=\Bigl\{a\in A\st |X_a|\ge \tfrac{(k-1)\eps}k\cdot M^{k-1}\Bigr\}\,.
\]
Due to
\[
\eps M^k\le |X|=\sum_{a\in A}|X_a|\le 
\tfrac{(k-1)\eps}k\cdot M^{k}+|A'|\cdot M^{k-1}\,,
\]
we have $|A'|\ge \eps M/k$ and we can select a subset $A''\subseteq A'$ with
$|A''|=\lceil \eps M/k\rceil=:m$. Moreover, for every $a\in A''$ we may apply
the induction hypothesis to $X_a\subseteq S_a$ with $\eps'=(k-1)\eps/k$, thus obtaining a $[k-1, m]$-system
$S'_a\subseteq X_a$. Consequently,
\[
S'=\{a\circ b\st a\in A'' \text{ and } b\in S'_a\}
\]
is a $[k,m]$-system contained in $X$. 
\end{proof}

Next we introduce the somewhat technical notion of a \emph{fortress} 
(see Definition~\ref{def:fortress} below),
in a reduced hypergraph $\cA$. The additional structural requirements 
for $[k,M]$-systems to support a fortress serve two purposes: firstly 
a $[r,2]$-system that supports a fortress will give rise to a clique of 
order $2^r$ in $\cA$ and secondly $[r,M]$-systems for $M\geq 2$ that support
fortresses will be ``rich enough'' for the intended inductive arguments.

Consider an $M$-ary tree $T$ of height $k$ and the associated $[k, M]$-system $[T]$. 
For every sequence $c=(c_1, \ldots, c_\ell)\in T$ we
set
\[
	Q(c)=\bigl\{(d_1, \ldots, d_\ell)\st 
	d_i\in\ccS_T(c|(i-1))\setminus \{c_i\}\ \text{for every $i\in[\l]$}\bigr\}\,.
\] 
Since $T$ is an $M$-ary tree, there are $M-1$ successors of $c|(i-1)$ 
different from $c_i$ for every $i\in[\l]$ and thus we have $|Q(c)|=(M-1)^{|c|}$ 
for each $c\in T$.
Moreover, it follows from the definition that for the empty sequence $\vn$ we have $Q(\vn)=\{\vn\}$.
We also remark that $Q(c)$ is not necessarily a subset of $T$. For example, 
if $(a,\alpha_1), (a,\alpha_2), (a,\alpha_3), (b,\beta_1),\dots, (c,\gamma_3)$ are the leaves of 
a ternary tree of height two, then $Q((b,\beta_2))$ consists of $(a,\beta_1), (a,\beta_3), (c,\beta_1)$, and 
$(c,\beta_3)$. In the next definition we will make use of the fact, that if $d\in Q(c)$, then $d|s\in Q(c|s)$ for any $s=0,\dots,|c|$.

\begin{dfn}[fortress]\label{def:fortress} 
Let $T$ be an $M$-ary tree of height $k$ and let 
$\cA$ be a reduced hypergraph whose index set $I$ contains the $[k,M]$-system $S=[T]$. 
We say \emph{$S$ supports a fortress in $\cA$} if for every $a$, $b\in S$ with $a\neq b$
and for every $d\in Q(a\wedge b)$ there exists some vertex $P^{ab}_d\in\cP^{ab}\subseteq V(\cA)$\footnote{To be consistent with the notation in Section~\ref{sec:regmethod} we should maybe write something like $P^{ab}_{\alpha_{ab}(d)}$ where $1\leq \alpha_{ab}(d) \leq |\cP^{ab}|$. However, for a simpler notation  we will suppress such functions $\alpha_{ab}\colon Q(a\wedge b)\to[|\cP^{ab}|]$
and  simply write $P^{ab}_d$.}
such that 
\begin{enumerate}[label=\Flabel]
\item\label{it:F2}
for all distinct $a,b,c\in S$ satisfying
\begin{equation}\label{eq:ffort}
s:=|a\wedge b|=|a\wedge c|<|b\wedge c|\,,
\end{equation}
and for every $d\in Q(b\wedge c)$ with $d(s+1)=a(s+1)$ we have
\[
\bigl\{P_{d|s}^{ab}, P_{d|s}^{ac}, P_d^{bc}\bigr\}\in E(\cA^{abc})\,.
\]
\end{enumerate}
We refer to the set of vertices $\cF=\{P^{ab}_d\colon a,b\in S,\ a\neq b,\tand d\in Q(a\wedge b)\}$
as a \emph{fortress} and 
we say that $\cA$ contains a {\it $[k, M]$-fortress} if some subset of $I$ is a 
$[k, M]$-system supporting a fortress.
\end{dfn}

\begin{rem}\label{rem:1fortress}
Note that property~\ref{it:F2} is void for any $[1,M]$-fortress, since for a
$[1,M]$-system~$S$ we have
$a\wedge b=\vn$ for any two distinct $a$, $b\in S$. Therefore,~\eqref{eq:ffort}
will never hold for distinct $a$, $b$, $c\in S$ and we can select $P^{ab}_\vn\in\cP^{ab}$
arbitrarily.
\end{rem}

As mentioned above we now show that for $r\geq 2$ a $[r,2]$-fortress yields a clique of order~$2^r$ in $\cA$.
\begin{fact}\label{fact:fortress}
For every integer $r\geq 2$ a reduced hypergraph $\cA$ contains a $[r,2]$-fortress if and only if it contains a
clique of order $2^r$.
\end{fact}
\begin{proof}
	For a binary tree $T$ of height $r$ we have $|Q(c)|=1$ for any $c\in T$. Consequently, 
	an $[r,2]$-fortress contained in $\cA$  corresponds to a subset $J$ of the index set of $\cA$
	of size $|J|=2^r$
	and a selection $\{P^{ab}\in\cP^{ab}\colon a,b\in J\tand a\neq b\}$ such that 
	$\{P^{ab}, P^{ac}, P^{bc}\}\in E(\cA)$, whenever~\eqref{eq:ffort} holds.
	In fact, the condition $d(s+1)=a(s+1)$ for the unique $d\in Q(b\wedge c)$
	follows for binary trees directly from~\eqref{eq:ffort}.
	Moreover, any three distinct leaves of a binary tree can be labeled $a$, $b$,~$c$ 
	in such a way that~\eqref{eq:ffort}  holds and, consequently, 
	an $[r,2]$-fortress corresponds to a clique of order~$2^r$ in $\cA$.
	
	On the other hand, if $\cA$ with index set $I$
	contains a clique of order $2^r$, then there exists a 
	subset $J\subseteq I$ of size $2^r$ and vertices $P^{ij}\in\cP^{ij}$
	for all distinct $i$, $j\in J$ such that $\{P^{ij}, P^{ik}, P^{jk}\}$
	is a hyperedge of $\cA$ for all distinct $i$, $j$, $k\in J$.
	Relabelling all elements of $J$ by binary sequences of length $r$ 
	gives rise to a binary tree of height $r$ that carries a fortress 
	in $\cA$.
\end{proof}
\begin{rem}\label{rem:colour} 
It is not hard to show that a $[2,m]$-fortress in $\cA$ 
gives rise to a situation where the embedding lemma (Theorem~\ref{thm:EL}) can be applied 
for any hypergraph~$F$ on at most $m$ vertices 
with the property that one can colour the vertices and the pairs 
of vertices of $F$ with~$m$ colours such that every hyperedge $\{x,y,z\}$ of $F$ contains exactly 
two vertices, say $x$ and~$y$, which are coloured by the same colour, say red,
and the pair $\{x,y\}$ and the vertex $z$ have the same colour different from red.

Note that this includes for example all $2$-colourable hypergraphs~$F$ and in view of 
Proposition~\ref{lem:fortress} for $r=2$ 
this can be used to show that $\piee(F)=0$ holds
for any $2$-colourable $3$-uniform hypergraph $F$. In fact, for general $r$
one can show that $[r,m]$-fortresses allow the embedding of $2^{r-1}$-partite 
$3$-uniform hypergraphs $F$ with $m$ vertices
and, as a result, one can deduce $\piee(F)=\frac{r-2}{r-1}$ for such $F$. We omit
the details here.
\end{rem}

Proposition~\ref{lem:reduced} follows by Fact~\ref{fact:fortress}  from the case $m=2$ 
of Proposition~\ref{lem:fortress}.

\begin{prop}\label{lem:fortress} 
Suppose that integers $r, m\ge 2$ and a real $\eps>0$ are given. Then there are
a real $\delta>0$ and an integer $M$ with the property that every 
$\bigl(\tfrac{r-2}{r-1}+\eps, \delta, \ee\bigr)$-dense reduced hypergraph 
whose index set is an $[r, M]$-system contains an $[r, m]$-fortress.   
\end{prop}

The proof of Proposition~\ref{lem:fortress} in turn proceeds in $r$ steps. The first idea of this kind 
one might come up with is to wish proving by induction on $k$ that for $k\in [r]$ and 
$m\ll M\ll \delta^{-1}$ every $\bigl(\tfrac{r-2}{r-1}+\eps, \delta, \ee\bigr)$-dense reduced hypergraph, whose index set $I$ is a $[k, M]$-system, contains a $[k, m]$-fortress. However, for $k<r$ we have 
$\frac{k-2}{k-1}<\frac{r-2}{r-1}$ and 
Proposition~\ref{lem:fortress} asserts, that a $[k,m]$-fortress already appear 
in $\bigl(\tfrac{k-2}{k-1}+\eps, \delta, \ee\bigr)$-dense reduced hypergraphs. This seems to indicate 
that for $k<r$ we can insist on additional side-conditions, which will be utilised in the inductive step
and that become weaker for larger $k$.

The condition that turned out to work for us says roughly
the following: Suppose that the index set $I$ contains, besides the $[k, M]$-system $X_0$ under 
discussion, also some further sets $X_1, \ldots, X_{r-k}$ that we know to be ``well-attached'' to $X_0$
in the sense that we are given for every $x\in X_0$ and every $y\in \bigcup_{j\in[r-k]}X_j$
a vertex $P^{xy}\in \cP^{xy}$ such that for all  choices $x,x'\in X_0$ and $y\in \bigcup_{j\in[r-k]}X_j$
the vertices $P^{xy}$ and $P^{x'y}$ have high pair-degree in $\cP^{xx'}$ (see Definition~\ref{def:admissible} below). Note that this property would 
be given automatically for any choice of $P^{xy}$, if there would be no exceptional pairs in $\cA$ in the sense of 
Definition~\ref{def:denseA}.

This ``well-attachedness'' allows us to shrink the sets $X_1, \ldots, X_{r-k}$
down to linearly-sized subsets $Y_1,\dots,Y_{r-k}$ of themselves, 
such that later on one can find the vertices $P^{xx'}_d$ of the  
desired $[k, m]$-fortress 
in the neighbourhood of $P^{xy}$ and $P^{x'y}$ 
for all $y\in  \bigcup_{j\in[r-k]}Y_j$.  
This additional property will be crucial for the inductive construction of the fortress in the 
proof of Proposition~\ref{lem:general}.   
    
\begin{dfn}\label{def:admissible}
Let~$X$ and $Y$ be two disjoint subsets of the index set of a reduced hypergraph~$\cA$ and 
suppose that $d\in[0,1]$. A {\it $d$-admissible 
$(X, Y)$-selection} is a collection 
\[
\ccC=\{P^{xy}\in \cP^{xy}\st x\in X\text{ and }y\in Y\}
\]
of vertices of $\cA$ such that
\begin{enumerate}
\item[$\bullet$]
if $x, x'\in X$ are distinct and $y\in Y$, then the pair-degree of $P^{xy}$ and $P^{x'y}$ 
in $\cP^{xx'}$ is at least $d\,|\cP^{xx'}|$.
\end{enumerate} 
\end{dfn}
We remark that this notion of an admissible selection is not symmetric in $X$ and $Y$, i.e., a
$d$-admissible 
$(X, Y)$-selection $\ccC$ is not necessarily also a $d$-admissible $(Y, X)$-selection.

Our next immediate objective is to formulate and prove the first step of the induction for $k=1$
(see Lemma~\ref{lem:base} below) of the upcoming Proposition~\ref{lem:general}.
As it turns out the case $k=1$ is a bit simpler than the case of general $k$, because 
(i) the arising constants can be calculated rather easily, 
(ii) one of the assumptions and consequently of the variables turn out to be unnecessary 
when $k=1$, and
(iii) the notion of a $[1, m]$-fortress is especially simple. 
However, for later purposes it is better to prove a probabilistic strengthening of 
the statement for $k=1$. For all these reasons,
we deal with this case separately.  

\begin{lemma}\label{lem:base}
Let integers $r, m\ge 2$ and a real $\eps>0$ be given. 
Assume further 
\begin{enumerate}
\item[$\bullet$]
that $X_0, X_1, \ldots, X_{r-1}$ are disjoint subsets of the index
set of a reduced hypergraph~$\cA$ with $|X_0|=m$, 
\item[$\bullet$]
and that $\ccC_j=\{P^{xy}\st x\in X_0 \text{ and } y\in X_j\}$ is an
$\bigl(\tfrac{r-2}{r-1}+\eps\bigr)$-admissible $(X_0, X_j)$-selection for every $j\in [r-1]$.
\end{enumerate} 
For a collection of vertices
\[
	\ccC=\bigl\{P^{xx'}\in\cP^{xx'}\st x, x'\in X_0 \text{ and } x\ne x'\bigr\}
\]
and $j\in[r-1]$ set
\[
	Y_j(\ccC)=\big\{y\in X_j\st\{P^{xx'}, P^{xy}, P^{x'y}\}\in E(\cA^{xx'y})\text{ holds for all distinct }
	x, x'\in X_0\big\}\,.
\]
Then for  a selection $\ccC$ chosen uniformly at random from $\prod_{x\neq x'\in X_0}\cP^{xx'}$
the events
\[
	|Y_j(\ccC)|\geq (\tfrac\eps2)^{\binom{m}{2}}\,|X_j|
\]
hold simultaneously for all $j\in[r-1]$ with probability at least $(\tfrac\eps2)^{\binom{m}{2}}$.
\end{lemma}

\begin{proof}
We commence by treating the case $m=2$ and $X_0=\{x, x'\}$, say. For each vertex 
$P^{xx'} \in\cP^{xx'}$ and each $j\in[r-1]$ we set
\[
Y_j(P^{xx'})=\bigl\{y\in X_j\st \{P^{xx'}, P^{xy}, P^{x'y}\}\in E(\cA^{xx'y})\bigr\}\,. 
\]
Notice that for every fixed $j\in [r-1]$ we have
\[
\sum_{P^{xx'}\in\cP^{xx'}}|Y_j(P^{xx'})|\ge \bigl(\tfrac{r-2}{r-1}+\eps\bigr)\,|\cP^{xx'}|\,|X_j|
\]
owing to our admissibility assumption, and thus the set 
\[
A_j=\{P^{xx'}\in \cP^{xx'}\st |Y_j(P^{xx'})|\ge \tfrac\eps2\,|X_j|\}
\]
satisfies $|A_j|\ge\bl \tfrac{r-2}{r-1}+\tfrac\eps2\br \cdot|\cP^{xx'}|$. 
Consequently for their intersection $A=\bigcap_{j\in[r-1]}A_j$ we obtain
$|A|\ge \tfrac{(r-1)\eps}{2}\cdot|\cP^{xx'}|\ge \tfrac\eps2\,|\cP^{xx'}|$.
Hence, when $P^{xx'} \in\cP^{xx'}$ gets chosen uniformly at random, the event
$P^{xx'}\in A$ happens with probability at least $\tfrac\eps2$. 
Thereby the case $m=2$ of our lemma is proved.

To obtain the general case we iterate this argument $\binom{m}{2}$ many times. 
This means that we make a list $e_1, \ldots, e_{\binom{m}{2}}$ of the two-element subsets 
of $X_0$, say $e_i=\{x_i, x'_i\}$. Now imagine that rather than picking the vertices 
$\bigl\{P^{x_ix_i'}\in\cP^{x_ix_i'}\st i\in \binom{m}{2}\bigr\}$ simultaneously we would pick
them one by one, each choice being uniformly at random and independent from all previous 
choices. For $h=0,\dots,\binom{m}{2}$ let $\ccC^h=\{P^{x_1x'_1},\dots,P^{x_hx'_h}\}$ and set
\[
Y_j(\ccC^h)=\big\{y\in X_j\st 
\{P^{x_ix_i'}, P^{x_iy}, P^{x'_iy}\}\in E(\cA^{x_ix'_iy}) 
\text{ holds for all } i\in[h]
\bigr\}\,.
\]
Thereby we get for every $j\in[r-1]$ a sequence of sets
\[
X_j=Y_j(\ccC^0)\supseteq Y_j(\ccC^1)\supseteq\ldots\supseteq Y_j(\ccC^{\binom{m}{2}})=Y_j\,.
\]  
By the case $m=2$ of our lemma, for every $h\in \big[\binom{m}{2}\bigr]$ 
the event
\[
\cE^h=\bigl\{\text{for every } j\in[r-1] \text{ we have } |Y_j(\ccC^h)|\ge\tfrac\eps2\,|Y_j(\ccC^{h-1})|
\bigr\}
\]
has the property that
\[
\PP\bigl(\cE^h\,|\,\ccC^{h-1}\bigr)\ge\tfrac\eps2
\]
holds for every fixed choice of $P^{x_1, x'_1}, \ldots, P^{x_{h-1}, x'_{h-1}}$. 
It follows that the event $\cE$ that all the events 
$\cE^1, \ldots, \cE^{\binom{m}{2}}$ happen 
has at least the probability~$(\tfrac\eps2)^{\binom{m}{2}}$. 
Moreover, since $\cE$ implies for every $j\in[r-1]$
\[
	|Y_j|\ge (\tfrac\eps2)^{\binom{m}{2}}\,|X_j|\,,
\]
we are thereby done.
\end{proof}

The next and final Proposition tells what we can achieve in the $k$-th step of the proof of 
Proposition~\ref{lem:fortress}. In particular, for the special case $k=r$ the second items in the 
assumption and in the conclusion  of Proposition~\ref{lem:general} are void and the statement 
coincides with Proposition~\ref{lem:fortress}.

\begin{prop}\label{lem:general}
Given any integers $r, m\ge 2$, some $k\in [r]$, and a real $\eps>0$, 
there exists an integer $M$ and reals $\delta, \eta>0$ such that the 
following holds: 

Suppose 
\begin{enumerate}
\item[$\bullet$]
that $X_0, X_1, \ldots, X_{r-k}$ are disjoint subsets of the index set of some 
$\bigl(\tfrac{r-2}{r-1}+\eps, \delta, \ee\bigr)$-dense reduced hypergraph $\cA$,
where $X_0$ is a $[k, M]$-system,
\item[$\bullet$]
and that $\ccC_j=\{P^{xy}\st x\in X_0 \text{ and } y\in X_j\}$ is an 
$\bigl(\tfrac{r-2}{r-1}+\eps\bigr)$-admissible $(X_0, X_j)$-selection for $j\in[r-k]$.
\end{enumerate}
Then there are 
\begin{enumerate}
\item[$\bullet$]
a $[k,m]$-subsystem $Z_0\subseteq X_0$ carrying a $[k,m]$-fortress
\[
\cF=\bigl\{P^{zz'}_d\st z, z'\in Z_0, z\ne z', \text{ and } d\in Q(z\wedge z')\bigr\}
\]
\item[$\bullet$]
and sets $Y_j\subseteq X_j$ with $|Y_j|\ge \eta\,|X_j|$ for $j\in [r-k]$ 
\end{enumerate}
such that for every distinct $z, z'\in Z_0$, $d\in Q(z\wedge z')$, and
$j\in[r-k]$ we have 
\begin{equation}\label{eq:goal}
\{P^{zz'}_d, P^{zy}, P^{z'y}\}\in E(\cA^{zz'y})\quad\text{for every $y\in Y_j$,}
\end{equation}
where $P^{zy}$ and $P^{z'y}$ are given by $\ccC_j$.
\end{prop}

\begin{proof}We consider $r$ and $\eps$ to be fixed and proceed by induction on $k$. 

To deal with the base case $k=1$ let $m\geq 2$ be given and set $M=m$. Moreover, 
set $\eta=(\tfrac\eps2)^{\binom{m}{2}}$. In this case the 
density assumption of $\cA$ will not be utilised and, hence, $\delta>0$ can be chosen arbitrarily.
Then we are given disjoint subsets $X_0$, $X_1,\dots,X_{r-1}$ of the index set of~$\cA$, where 
$X_0$ forms a $[1,M]$-system,
and $\bigl(\tfrac{r-2}{r-1}+\eps\bigr)$-admissible $(X_0, X_j)$-selections $\ccC_j$ for 
$j\in[r-1]$.

Set $Z_0=X_0$ and by Lemma~\ref{lem:base} applied with $r$, $m$, $\eps$, 
$X_0$, $X_1,\dots,X_{r-1}$, and $\ccC_1,\dots,\ccC_{r-1}$ there exists a collection $\ccC$ of
vertices $P_\vn^{zz'}\in\cP^{zz'}$ for distinct $z, z'\in X_0$ and subsets $Y_j\subseteq X_j$ 
with $|Y_j|\ge \eta\,|X_j|$ for $j\in [r-1]$ such that for every $j\in[r-1]$
we have \[\{P_\vn^{zz'}, P^{zy}, P^{z'y}\}\in E(\cA^{zz'y})\] 
for all distinct $z, z'\in X_0$, $y\in Y_j$, and $P^{zy}, P^{z'y}\in \ccC_j$. 

Owing to $k=1$, the collection
\[
\cF=\bigl\{P_\vn^{zz'}\st z, z'\in X_0 \text{ and } z\ne z'\bigr\}
\]
is a $[1,m]$-fortress on $X_0$ for trivial reasons
(see Remark~\ref{rem:1fortress}). As stated before, $\cF$ has the required property
and this establishes the induction start.

For the induction step we suppose that $2\le k\le r$, and that the proposition is valid for $k-1$
in place of $k$. Whenever we apply this case of Proposition~\ref{lem:general} 
to some integer $m\ge2$ it returns 
an integer $M(m)$ and two positive reals called $\delta(m)$ and $\eta(m)$. 

From now on, we fix an integer $m\ge 2$ for which we would like to complete the induction step.
We divide the argument into five parts.

\medskip

\noindent
{\bf Part I. Choice of the constants.} To begin with, we define a decreasing sequence of
integers $M_0, \ldots, M_{m(m-1)}$ by backwards induction as follows:
\begin{equation}\label{eq:M-intro}
M_{m(m-1)}=m 
\qquad\text{ and }\qquad
M_{h-1}=M(M_h)+\Biggl\lceil \frac{(k-1)M_h}{\eta(M_h)}\Biggr\rceil
\quad \text{ for } h\in[m(m-1)]\,.
\end{equation}
We set $M=M_0$ and define
\begin{equation}\label{eq:del-intro}
\eta_{0}=(\tfrac\eps2)^{m^2M^{2(k-1)}} \qquad\text{ and }\qquad \eta_{h}=\eta_{h-1}\cdot \eta(M_h)
\quad \text{ for } h\in[m(m-1)]\,.
\end{equation}
Finally, let
\begin{equation}\label{eq:eta-intro}
\delta=\min\left(\bigl\{m^{-2}M^{-3(k-1)}\eta_0\bigr\}\cup\bigl\{\delta(M_h)\st h\in[m(m-1)]\bigr\}\right)
\tand \eta=\eta_{m(m-1)}\,.
\end{equation}
We shall prove the proposition for this choice of the constants $M$, $\delta$, and $\eta$.

\medskip

\noindent
{\bf Part II. The first level of the $[k,m]$-tree underlying $Z_0$.}  
Now let $\cA$, $X_0,\ldots, X_{r-k}$ as well as the 
$\bigl(\tfrac{r-2}{r-1}+\eps\bigr)$-admissible $(X_0, X_j)$-selections 
$\ccC_j=\{P^{xy}\in\cP^{xy}\st x\in X_0 \text{ and } y\in X_j\}$ for $j\in [r-k]$ 
be as described in the statement of the proposition. We denote the underlying tree of the 
$[k, M]$-system $X_0$ by~$T$. Let 
\[
A\subseteq\bigl\{a\in T\st |a|=1\bigr\} \ \text{of size}\ |A|=m
\]
be an arbitrary subset. We intend to construct the $[k,m]$-tree underlying $Z_0$ in such a way 
that the direct continuations of its root is $A$. 

For every $a\in A$ we set
\[
X_0^a=\{z\st a\circ z\in X_0\}
\]
and note that these sets are $[k-1, M]$-systems. In Part~IV 
we shall apply the induction hypothesis (several times) to $X_0^a$ for every $a\in A$, which then 
will lead to the desired $[k,m]$-system $Z_0$. 

Strictly speaking, $X_0^a$ is not a subset of the index set $I$ of $\cA$, which would be required 
for the application of the induction hypothesis. However, in such situations we can
simply identify~$X_0^a$ with $\{\tilde{x}\in X_0\st \tilde{x}|1=a\}\subseteq I$ to
circumvent this technicality and below we will suppress this identification.

\medskip

\noindent
{\bf Part III. The selection of some vertices  $P_\vn^{a\circ z, a'\circ z'}$.} 
Our next objective is to select the elements~$P_\vn^{a\circ z, a'\circ z'}$, i.e., those labeled with $\vn$,
of the desired fortress $\cF$.
In view of the conclusion of Proposition~\ref{lem:general} we shall choose 
\begin{enumerate}[label=\rmlabel]
\item\label{it:3i}for any distinct $a, a'\in A$, $z\in X_0^a$, and $z'\in X_0^{a'}$
a vertex $P_\vn^{a\circ z, a'\circ z'}\in\cP^{a\circ z, a'\circ z'}$
such that for all distinct $a, a'\in A$ the set
\[
\bigl\{P_\vn^{a\circ z, a'\circ z'}\st z\in X_0^a \text{ and } z'\in X_0^{a'}\bigr\} 
\]
is a $\bigl(\tfrac{r-2}{r-1}+\eps\bigr)$-admissible $(X_0^a, X_0^{a'})$-selection.
\item\label{it:3ii}and subsets $Y^0_j\subseteq X_j$ with $|Y^0_j|\ge\eta_0\,|X_j|$ for $j\in[r-k]$
\end{enumerate}
such that for all distinct $a, a'\in A$, $z\in X_0^a$, $z'\in X_0^{a'}$ and $j\in[r-k]$ we have 
\begin{equation}\label{eq:Y0j}
\bigl\{P_\vn^{a\circ z, a'\circ z'}, P^{a\circ z, y}, P^{a'\circ z', y}\bigr\}
\in E(\cA^{a\circ z, a'\circ z', y})\quad\text{for every $y\in Y^0_j$,}
\end{equation}
where $P^{a\circ z, y}$ and $P^{a'\circ z', y}$ are given by $\ccC_j$.

We choose $P_\vn^{a\circ z, a'\circ z'}\in\cP^{a\circ z, a'\circ z'}$ independently 
and uniformly at random and below we show that property~\ref{it:3i} holds with probability bigger than $1-\eta_0$
and property~\ref{it:3ii} is satisfied with probability at least $\eta_0$.

Dealing with property~\ref{it:3i} first, we consider the set
\[
K=\bigl\{ (\{a\circ z, a\circ z'\}, b\circ y) \st 
a, b\in A, \,\, a\ne b, \,\, z, z'\in X_0^a, \,\, z\ne z', \text{ and } y\in X_0^b\bigr\}
\]
consisting of those combinations of indices for which property~\ref{it:3i} has to be checked.
Since $|X_0^a|=M^{k-1}$ for all $a\in A$ and $|A|=m$, we have
\begin{equation}\label{eq:sizeK}
|K|=m(m-1)\binom{M^{k-1}}{2}M^{k-1}<m^2M^{3(k-1)}\,.
\end{equation}
Moreover, in view of the $\bigl(\tfrac{r-2}{r-1}+\eps, \delta, \ee\bigr)$-denseness of $\cA$
for each $(\{a\circ z, a\circ z'\}, b\circ y)\in K$ the ``bad event'' $\cE_{abzz'y}$ that the 
pair-degree of $P_\vn^{a\circ z, b\circ y}$ and $P_\vn^{a\circ z', b\circ y}$ in 
$\cP^{a\circ z, a\circ z'}$ is smaller than 
$\bigl(\tfrac{r-2}{r-1}+\eps\bigr)\,|\cP^{a\circ z, a\circ z'}|$ has probability 
  at most $\delta$. 
So the union bound together with \eqref{eq:sizeK} and \eqref{eq:eta-intro} yields
\begin{equation}\label{eq:Pradmis}
	\PP\bigl(\cE_{abzz'y}\ \text{occurs for some } (\{a\circ z, a\circ z'\}, b\circ y)\in K\bigr)
	\le |K|\delta 
	< \eta_0\,,
\end{equation}
which shows that with probability greater than $1-\eta_0$ the random selection satisfies 
property~\ref{it:3i}.

Next we turn to property~\ref{it:3ii}.
For this it is convenient to introduce the set
\[
L=\bigl\{\{a\circ z, a'\circ z'\}\st a, a'\in A, \,\, a\ne a',\,\, z\in X^a_0,\text{ and } z'\in X^{a'}_0      \bigr\}\,.
\]
For $j\in[r-k]$ we consider the random subsets 
\[
Y^0_j=\big\{y\in X_j\st\{P^{a\circ z,a'\circ z'}_\vn, P^{a\circ z, y}, P^{a'\circ z',y}\}\in E(\cA^{zz'y})
\text{ holds for all } \{a\circ z, a'\circ z'\}\in L\big\}\,,
\] 
where the randomness is induced by the random choice of $P^{a\circ z,a'\circ z'}_\vn\in\cP^{a\circ z,a'\circ z'}$ above.

Now we apply Lemma~\ref{lem:base} with $r-k+1$, $mM^{k-1}$ and $\bigcup_{a\in A}X_0^a$ in place of $r$, 
$m$, and $X_0$ and with $X_j$ and  $\ccC'_j=\{P^{a\circ x,y}\in \ccC_j\st a\in A\,,\ x\in X_0^a,\tand y\in X_j\}$
for $j\in[r-k]$.
It follows from our choice of $\eta_0$ in
\eqref{eq:del-intro} that the event
\[
\cE=\bigl\{|Y^0_j|\ge\eta_0\,|X_j| \ \text{for every}\ j\in[r-k] \bigr\} 
\]
holds with probability at least $\eta_0$. In fact, Lemma~\ref{lem:base} is more general, but at this 
point we only need it in this simpler form. 

Now $\PP(\cE)\ge \eta_0$ combined with~\eqref{eq:Pradmis} implies that there are 
the vertices $P^{a\circ z,a'\circ z'}_\vn\in\cP^{a\circ z,a'\circ z'}$ satisfying
both properties~\ref{it:3i} and~\ref{it:3ii} promised above.

\medskip

\noindent
{\bf Part IV. Inductive construction of subfortresses.} Let $e_i=(a_i, b_i)$ for  
$i\in \mbox{$[m(m-1)]$}$ enumerate all ordered pairs of distinct elements from $A$. We will show
that for every nonnegative integer $h\le m(m-1)$ the following statement is true:

\begin{description}
\item[$(*)_h$]
There are 
\begin{enumerate}
\item[$\bullet$]
for each $a\in A$ a sequence of sets 
\[
Z^{a,h}\subseteq Z^{a, h-1}\subseteq \ldots \subseteq Z^{a, 0}=X^a_0
\]
with $Z^{a, i}$ being a $[k-1, M_i]$-system for every nonnegative $i\le h$,
\item[$\bullet$]
and for each $j\in[r-k]$ a subset $Y^h_j\subseteq Y^0_j$ with $|Y^h_j|\ge \eta_h\, |X_j|$,
\end{enumerate}
such that for each $i\in [h]$ the $[k-1,M_h]$-system $Z^{a_i, h}$ carries a $[k-1,M_h]$-fortress
\[
{\cF^{\,i, h}}=\bigl\{P^{a_i\circ z,\,a_i\circ z'}_{b_i\circ d}\st z, z'\in Z^{a_i, h}, 
z\ne z', \text{ and } d\in Q(z\wedge z')\bigr\}
\]
such that for every distinct $z, z'\in Z^{a_i, h}$ and $d\in Q(z\wedge z')$
we have
\begin{align}\label{eq:F2-0}
\bigl\{P^{a_i\circ z,\,a_i\circ z'}_{b_i\circ d}, P^{a_i\circ z, \,b_i\circ w}_\vn, 
P^{a_i\circ z', \,b_i\circ w}_\vn\bigr\}&\in 
E\bigl(\cA^{a_i\circ z, \,a_i\circ z', \,b_i\circ w}\bigr) \quad \text{for every $w\in Z^{b_i, h}$} \\
\intertext{and for every $j\in[r-k]$ we have} 
\label{eq:*h2}
\bigl\{P^{a_i\circ z,\,a_i\circ z'}_{b_i\circ d}, P^{a_i\circ z, \,y}, 
P^{a_i\circ z', \,y}\bigr\}&\in E\bigl(\cA^{a_i\circ z, \,a_i\circ z', \,y}\bigr)\quad \text{for every $y\in Y^h_j$}\,.
\end{align}
\end{description}

To show this we argue by induction on $h$. In the base case $h=0$ we have to take 
$Z^{a,0}=X^a_0$ for all $a\in A$ and the sets $Y^0_j$ obtained in Part~III. 
The assertion about the existence of fortresses holds vacuously. 

Now suppose that some $h\in [m(m-1)]$ is such that $(*)_{h-1}$ holds with
$Z^{a,h-1}\subseteq X^a_0$ for $a\in A$, with $Y^{h-1}_j\subseteq Y_j^0$ for $j\in[r-k]$, and 
with the $[k-1,M_i]$-fortresses $\cF^{\,i, h-1}$ for $i\in [h-1]$. Now we apply the outer induction hypothesis from the proof of Proposition~\ref{lem:general} with $M_h$ in place of $m$
\begin{enumerate}
\item[$\bullet$]
to the $[k-1, M_{h-1}]$-system $Z^{a_h, h-1}$ (in place of $X_0$) and the $r-(k-1)$ further subsets 
$Z^{b_h, h-1}, Y^{h-1}_1, \ldots, Y^{h-1}_{r-k}$ of $I$ (in place of $X_1,\dots, X_{r-(k-1)}$),
\item[$\bullet$]
to the $\bigl(\tfrac{r-2}{r-1}+\eps\bigr)$-admissible $(Z^{a_h,h-1},Z^{b_h,h-1})$-selection
\[
\bigl\{P_\vn^{a_h\circ z,b_h\circ z'} \st z\in Z^{a_h,h-1} 
\text{ and } z'\in Z^{b_h,h-1}\bigr\} 
\]
obtained in Part III, and for $j\in[r-k]$ to the $\bigl(\tfrac{r-2}{r-1}+\eps\bigr)$-admissible $(Z^{a_h,h-1},Y^{h-1}_j)$-selections
\[
\bigl\{P^{a_h\circ z, y} \st z\in Z^{a_h,h-1} 
\text{ and } y\in Y^{h-1}_j\bigr\}\subseteq \ccC_j 
\]  
provided by the assumption.
\end{enumerate}
Since $M_{h-1}\ge M(M_h)$ holds by~\eqref{eq:M-intro}, this yields in particular 
\begin{enumerate}
\item[$\bullet$]
a $[k-1, M_h]$-system $Z^{a_h, h}\subseteq Z^{a_h, h-1}$ carrying a $[k-1,M_h]$-fortress
\[
{\cF^{\,h, h}}=\bigl\{P^{a_h\circ z,\,a_h\circ z'}_{b_h\circ d}\st z, z'\in Z^{a_h, h}, 
z\ne z', \text{ and } d\in Q(z\wedge z')\bigr\}\,,
\]
\item[$\bullet$]
a subset $W\subseteq Z^{b_h, h-1}$ and subsets 
$Y^h_j\subseteq Y^{h-1}_j$ for $j\in[r-k]$ with 
\begin{equation}\label{eq:sizeWY}
	|W|\ge \eta(M_h)\,|M_{h-1}|^{k-1}\qand
	|Y^h_j|\ge \eta(M_h)\,|Y^{h-1}_j|\,,
\end{equation}
\end{enumerate}
such that for every distinct $z, z'\in Z^{a_h, h}$ and $d\in Q(z\wedge z')$ we have  
\begin{align}\label{eq:new-h1}
\bigl\{P^{a_h\circ z,\,a_h\circ z'}_{b_h\circ d}, P^{a_h\circ z, \,b_h\circ w}_\vn, 
P^{a_h\circ z', \,b_h\circ w}_\vn\bigr\}&\in 
E\bigl(\cA^{a_h\circ z, \,a_h\circ z', \,b_h\circ w}\bigr) \quad  \text{for every $w\in W$} \\
\intertext{and for every $j\in[r-k]$ we have} 
\label{eq:new-h2}
\bigl\{P^{a_h\circ z,\,a_h\circ z'}_{b_h\circ d}, P^{a_h\circ z,y}, 
P^{a_h\circ z',y}\bigr\}&\in E\bigl(\cA^{a_h\circ z,\,a_h\circ z',\,y}\bigr) \quad  \text{for every $y\in Y^h_j$.}
\end{align}   

Now we are ready to define the remaining entities verifying $(*)_h$. Recall that we have already 
obtained the $[k-1, M_h]$-system $Z^{a_h, h}$, the sets $Y^h_j$ for $j\in [r-k]$, and the
fortress~$\cF^{\,h, h}$. Note that~\eqref{eq:M-intro} yields
$\eta(M_h)M_{h-1}/(k-1)\ge M_h$. 

Thus Lemma~\ref{lem:system} applied with $k-1$ and $M_{h-1}$ 
to the $[k-1,M_{h-1}]$-system $Z^{b_h, h-1}$ and $W\subseteq Z^{b_h, h-1}$
which has size $\eta(M_h)|M_{h-1}|^{k-1}$ (see~\eqref{eq:sizeWY})
tells us that there exists a 
$[k-1, M_h]$-system $Z^{b_h, h}\subseteq W$. 

Finally, for definiteness (somewhat wastefully) for any
$c\in A\setminus\{a_h, b_h\}$ we let $Z^{c, h}$ be an arbitrary $[k-1, M_h]$-subsystem of
$Z^{c, h-1}$ and for $i\in [h-1]$ we let $\cF^{\,i, h}$ denote the ``restriction'' of $\cF^{\,i, h-1}$
to $Z^{a_i, h}$. 

It remains to check that we have indeed met all conditions mentioned in $(*)_h$. That 
$Z^{a, h}\subseteq Z^{a,h-1}$ holds for every $a\in A$ 
follows from our construction. Due to the choice of~$Y^h_j$, the description of $Y^{h-1}_j$ in $(*)_{h-1}$, and~\eqref{eq:del-intro} 
we have 
\[
|Y^h_j|\overset{\eqref{eq:sizeWY}}{\ge} \eta(M_h)\,|Y^{h-1}_j|\ge\eta(M_h)\eta_{h-1}\,|X_j|=\eta_h\,|X_j|
\]
for every $j\in [r-k]$. The statements~\eqref{eq:F2-0} and~\eqref{eq:*h2} hold for $i\ne h$ by
$(*)_{h-1}$ and for $i=h$ by~\eqref{eq:new-h1} and \eqref{eq:new-h2} respectively. 
This concludes the proof of $(*)_h$.

\medskip

\noindent
{\bf Part V. Conclusion of the argument.} We will show that the
statement $(*)_{m(m-1)}$ from Part~IV yields the conclusion of the proposition. 
For that set $Z^a=Z^{a, m(m-1)}$ for $a\in A$, 
and $Y_j=Y^{m(m-1)}_j$ for $j\in[r-k]$. It follows from $M_{m(m-1)}=m$ that $Z^a$ is a
$[k-1, m]$-system for each $a\in A$ and consequently the set
\[
Z_0=\bigr\{a\circ z\st a\in A \text{ and } z\in Z^{a}\bigr\}
\]
is a $[k,m]$-system. Moreover, a quick thought reveals that the collection of vertices
\begin{align*}
\cF=\bigl\{P_\vn^{a\circ z, a'\circ z'} & \st a, a'\in A, a\ne a', z\in Z^a\text{ and } z'\in Z^{a'}\bigr\} \\
& \cup
\bigl\{P_{b\circ d}^{a\circ z, a\circ z'}\st a,b\in A,\,\,  a\ne b, \,\, z, z'\in Z^a, \,\, 
z\ne z', \text{ and } d\in Q(z\wedge z')\bigr\}
\end{align*}
has the correct index structure for being a fortress on $Z_0$. To see that $\cF$ actually 
is a fortress we need to verify the axiom \ref{it:F2}; if $s=0$ it follows from~\eqref{eq:F2-0}
and for $s>0$ it follows from all the $\cF^{\,i, m(m-1)}$ with $i\in[m(m-1)]$ being fortresses.

We contend that the system $Z_0$, the fortress $\cF$, and the sets $Y_j$ with $j\in[r-k]$
are as desired. The latter are large enough because of $\eta=\eta_{m(m-1)}$ and
$(*)_{m(m-1)}$. Finally~\eqref{eq:goal} was obtained in~\eqref{eq:Y0j} for $z|1\ne z'|1$ and otherwise in~\eqref{eq:*h2}. This completes the induction step and thus the proof of
Proposition~\ref{lem:general}.
\end{proof}

\end{document}